\documentclass[12pt]{amsart}
\usepackage{a4,url,amssymb,amsmath,amsfonts}
\usepackage[colorlinks=true,urlcolor=blue,linkcolor=blue,citecolor=blue]{hyperref}

\numberwithin{equation}{section}
\theoremstyle{plain}
\newtheorem{thm}{Theorem}[section]
\newtheorem{lem}[thm]{Lemma}
\newtheorem{cor}[thm]{Corollary}

\theoremstyle{definition}

\newtheorem{exa}[thm]{Example}

\newtheorem{que}[thm]{Question}

\newcommand{\st}{\ \mid\ }

\newcommand{\A}{\ensuremath{\mathbf{A}}}
\newcommand{\N}{\ensuremath{\mathbb{N}}}

\newcommand{\Cg}{\mathrm{Cg}}

\title[Solvable, nilpotent and supernilpotent semigroups]{Solvable, nilpotent and supernilpotent semigroups with completely simple ideal and monoids}
\author[P.~Mayr]{Peter Mayr}
\address{Department of Mathematics, University of Colorado,  Boulder, USA}
\email{peter.mayr@colorado.edu}

 

\date{\today}

\begin{document}
\maketitle

\begin{abstract}
 Around 1980 commutator theory was generalized from groups to arbitrary algebras using the socalled term condition
 commutator. The semigroups that are abelian with respect to this commutator were classified by Warne~\cite{Wa:SOTC}.
 We study what solvability, nilpotence, and  supernilpotence in the sense of commutator theory mean for semigroups
 and how these notions relate to classical concepts in semigroup theory.
 We show that a semigroup with a completely simple ideal is solvable (left nilpotent or right nilpotent or
 supernilpotent) in the sense of commutator theory iff it is a nilpotent extension in the classical sense of semigroup
 theory of a completely simple semigroup with solvable (nilpotent) subgroups.
 These characterizations hold in particular for finite semigroups and for eventually regular semigroups,
 i.e., semigroups in which every element has some regular power.
 
 We also show that a monoid is (left and right) nilpotent in the sense of commutator theory iff it embeds into a
 nilpotent group.
\end{abstract}

\section{Introduction}
\label{sec:i}

 Commutators have been generalized from groups to other algebras in various ways.
 Around 1980 universal algebraists like Smith~\cite{Sm:MV}, Hermann~\cite{He:AAC}, Gumm~\cite{Gu:GMC},
 Freese and McKenzie~\cite{FM:CTC} developed what is now known as binary term condition commutator for congruences of
 algebras, which extends the commutator of two normal subgroups of groups. We refer to the latter citation for a
 historical overview, background and definitions (see also below).

 This commutator allows to generalize notions like abelian, solvable, nilpotent from groups to arbitrary algebras in a
 natural way. It is particularly well-behaved for algebras in varieties with modular congruence lattices
 (e.g. quasigroups, rings, lattices) and has proved extremely useful for the structure theory in this setting
 but also beyond. Unfortunately a variety of semigroups is congruence modular iff it consists of groups.
 Despite this, the commutator appeared in semigroup theory early on, for instance, in Warne's classification of abelian
 semigroups~\cite{Wa:SOTC,Wa:TCS}.

 In recent years, after the introduction of higher arity term condition commutators due to Bulatov~\cite{Bu:NFM}, 
 there has been renewed interest to understand and apply commutators to algebras outside of congruence modular varieties.
 Semigroups are natural candidates for these investigations as they are classical structures and their term operations
 are just induced by words, hence comparatively easy to understand.
 Mudrinski and Radovi{\'c} described higher commutators in semigroups with zero~\cite{RM:HCS},
 in completely simple semigroups~\cite{RM:CCS} and in orthodox semigroups~\cite{RM:SS}.
 Kinyon and Stanovsk{\'y} classified abelian and central congruence in inverse semigroups~\cite{KS:ACI}.

 We add to this line of research by first giving an overview of basic properties of the binary term condition commutator
 and higher commutators for semigroups.
 Then we investigate what solvability, nilpotence, and  supernilpotence in the sense of commutator theory mean for
 semigroups and how these notions relate to classical concepts in semigroup theory.
 Our main results are characterizations of solvable, nilpotent and supernilpotent for semigroups that have a
 completely simple ideal (Theorem~\ref{thm:csi}) and for eventually regular semigroups (Corollary~\ref{cor:er}),
 which unify those for completely simple, orthodox, and semigroups with zero given in~\cite{RM:HCS,RM:CCS,RM:SS},
 as well as a characterization of nilpotent (in the sense of commutator theory) monoids (Theorem~\ref{thm:nmonoid}).

\subsection{Definitions} We introduce basic notions and definitions of commutator theory based on~\cite{AM:SAHC,FM:CTC}.

 Let $\N := \{0,1,\dots \}$ and $[n] := \{1,\dots,n\}$ for $n\in\N$.

 We refer to~\cite{FM:CTC} for properties of the (binary) term condition commutator defined as follows:
 Let $\A$ be an algebra with congruences $\alpha,\beta,\delta$. Let
\[ M_\A(\alpha,\beta) := \left\langle \begin{bmatrix} a & a \\ b & b \end{bmatrix},\begin{bmatrix} c & d \\ c & d \end{bmatrix} \st (a,b)\in\alpha, (c,d)\in\beta \right\rangle \leq \A^{2\times 2}. \]  
 Then $\alpha$ \emph{centralizes} $\beta$ modulo $\delta$ (short $C(\alpha,\beta,\delta)$) if
\begin{equation} \label{eq:tc}
  \forall \begin{bmatrix} u & v \\ w & x \end{bmatrix}\in  M_\A(\alpha,\beta) : (u,v)\in\delta \Rightarrow (w,x)\in\delta.
\end{equation}  
 The \emph{commutator} $[\alpha,\beta]_\A$ is the smallest congruence $\delta$ of $\A$ such that $C(\alpha,\beta,\delta)$.
 Equivalently, $[\alpha,\beta]_\A$ is the smallest congruence $\delta$ such that each element in $M_\A(\alpha,\beta)$
 satisfies: if the entries in the first row are $\delta$-related, then the entries in the second row are $\delta$-related.

 We will usually omit the index $\A$ from $M_\A(\alpha,\beta)$ and $[\alpha,\beta]_\A$ if the algebra is clear from context.

 On an algebra $\A$ we denote the total congruence $A^2$ by $1$ and the trivial congruence (equality) by $0$.
 For a congruence $\alpha$ of $\A$ and $d\in\N$, define
 \[  [\alpha]^0 := \alpha,\ [\alpha]^{d+1} := [[\alpha]^d,[\alpha]^d]. \]
 Then $\A$ is
 \begin{itemize}
\item \emph{$d$-solvable} if $[1]^d = 0$;
\item \emph{left $d$-nilpotent} if $(1]^{d+1} := [\underbrace{1,[1,\dots [1,1}_{d+1}]\dots]]= 0$;
\item \emph{right $d$-nilpotent} if $[1)^{d+1} := [\dots\underbrace{[1,1],\dots,1}_{d+1}] = 0$; 
\item
 \emph{abelian} if $\A$ is $1$-solvable (equivalently left $1$-nilpotent or right $1$-nilpotent).
\end{itemize}
 Note that $[1]^0 = (1]^1 = [1)^1 = 1$ and $[1]^1 = (1]^2 = [1)^2 = [1,1]$.

 We will also need the notion of higher commutators as suggested by
 Bulatov in~\cite{Bu:NFM} and first developed by Aichinger and Mudrinski~\cite{AM:SAHC}.
 This generalizes the above binary commutator for congruences on algebraic structures from~\cite{FM:CTC}
 or~\cite[Section 4.13]{MMT:ALV1}.

 Let $\A$ be an algebra. For $n\in\N$ and congruences $\alpha_1,\ldots,\alpha_n$ of $\A$, we let
 $M_\A(\alpha_1,\ldots,\alpha_n)$ denote the subalgebra of
 $\A^{2^{n-1}\times 2}$ that is generated by all the elements 
 \[ \begin{tiny} \left(\begin{array}{cc}
    a_1 & a_1 \\
    \vdots & \vdots \\ 
    \vdots & \vdots \\ 
    \vdots & \vdots \\ 
    \vdots & \vdots \\ 
    \vdots & \vdots \\ 
    a_1 & a_1 \\
    b_1 & b_1 \\  
    \vdots & \vdots \\ 
    \vdots & \vdots \\ 
    \vdots & \vdots \\ 
    \vdots & \vdots \\ 
    \vdots & \vdots \\ 
    b_1 & b_1 \\  
    \end{array}\right),
 \left(\begin{array}{cc}
    a_2 & a_2 \\
    \vdots & \vdots \\ 
    \vdots & \vdots \\ 
    a_2 & a_2 \\
    b_2 & b_2 \\  
    \vdots & \vdots \\ 
    \vdots & \vdots \\ 
    b_2 & b_2 \\  
    a_2 & a_2 \\
    \vdots & \vdots \\ 
    \vdots & \vdots \\ 
    a_2 & a_2 \\
    b_2 & b_2 \\  
    \vdots & \vdots \\ 
    \vdots & \vdots \\ 
    b_2 & b_2 \\  
   \end{array}\right), \dots,
\left(\begin{array}{cc}
    a_{n-1} & a_{n-1} \\
    b_{n-1} & b_{n-1} \\
    a_{n-1} & a_{n-1} \\
    b_{n-1} & b_{n-1} \\
    \vdots & \vdots \\ 
    \vdots & \vdots \\ 
    \vdots & \vdots \\ 
    \vdots & \vdots \\ 
    \vdots & \vdots \\ 
    \vdots & \vdots \\ 
    \vdots & \vdots \\ 
    \vdots & \vdots \\ 
    \vdots & \vdots \\ 
    a_{n-1} & a_{n-1} \\
    b_{n-1} & b_{n-1} \\
   \end{array}\right),
\left(\begin{array}{cc}
    a_n & b_n \\
    \vdots & \vdots \\ 
    \vdots & \vdots \\ 
    \vdots & \vdots \\ 
    \vdots & \vdots \\ 
    \vdots & \vdots \\ 
    \vdots & \vdots \\ 
    \vdots & \vdots \\ 
    \vdots & \vdots \\ 
    \vdots & \vdots \\ 
    \vdots & \vdots \\ 
    \vdots & \vdots \\ 
    a_n & b_n \\  
    \end{array}\right)\in A^{2^{n-1}\times 2} \end{tiny} 
\]
 such that $(a_i,b_i)\in\alpha_i$ for all $i\in\{1,\ldots,n\}$.
 The $n$-\emph{ary commutator} $[\alpha_1,\ldots,\alpha_n]$ is defined as the
 smallest congruence $\delta$ of $\A$ such that for every
 $x\in M_\A(\alpha_1,\ldots,\alpha_n)$:
\begin{equation} \label{eq:htc}
 \text{if } (x_{i1}, x_{i2})\in\delta \text{ for all } i\in\{1,\ldots,2^{n-1}-1\}, \text{ then }  (x_{2^{n-1}1},x_{2^{n-1}2})\in\delta.
\end{equation}
 This definition is equivalent to the one given
 in~\cite[Definition~3.2]{AM:SAHC} and in~\cite[Definition~3]{Bu:NFM}.
 Note that $M_\A(\alpha_1) = \alpha_1$ for all congruences $\alpha_1$.
 For $n = 2$ the binary commutator defined by~\eqref{eq:htc} is precisely the one previously defined by~\eqref{eq:tc}.

 It is straightforward that the $n$-ary commutator is monotone in each
 argument and that it satisfies
\[ [\alpha_1,\ldots,\alpha_n]\leq\alpha_1\wedge\ldots\wedge\alpha_n \]
 and
 \[ [\alpha_1,\ldots,\alpha_n]\leq[\alpha_2,\ldots,\alpha_{n}]. \]
 Hence left or right $d$-nilpotent both imply $d$-solvable for arbitrary algebras and $d\in\N$. 
 We also have a descending chain of congruences
\[ 1 \geq [1,1] \geq [1,1,1] \geq \dots \]
 Let $d\in\N$. Then $\A$ is
\begin{itemize}
\item
 \emph{$d$-supernilpotent} if $[\underbrace{1,\ldots,1}_{d+1 \text{ times}}]= 0$.
\end{itemize}
 Again $1$-supernilpotent is abelian.
 
 For algebras in congruence modular varieties the (higher) commutators are also symmetric in their arguments and
 distribute over joins as proved by Moorhead~\cite{Mo:HCT}.
 In general, in particular for semigroups, these properties do not hold. Hence we have distinct notions of left and
 right nilpotence. For an explicit example, we will show the free monoid over $2$ generators is
 left $2$-nilpotent but neither right $2$-nilpotent nor $2$-supernilpotent in Theorem~\ref{thm:free}.
 
 Also the relationship between supernilpotence and nilpotence is more complicated than their names suggest: 
 A group is $d$-supernilpotent iff it is left $d$-nilpotent iff it is right $d$-nilpotent~\cite{AM:SAHC}.
 Moorhead showed that every supernilpotent algebra in a Taylor variety (including essentially every named algebra except
 $G$-sets and semigroups) is nilpotent~\cite{Mo:STA}.
 Kearnes and Szendrei proved that every finite supernilpotent algebra is nilpotent~\cite{KS:ISSN}.
 Still Moore and Moorhead constructed an infinite supernilpotent algebra which is not even solvable,
 hence neither left nor right nilpotent~\cite{MM:SNN}.
 It remains open whether supernilpotent semigroups are nilpotent (Question~\ref{que:sn} below).
 
 From the definitions it is straightforward that the class of left $d$-nilpotent algebras, etc., of fixed type
 is closed under direct products and subalgebras. We use this in the following example. 
   
\begin{exa} \label{exa:rectangular}
 (Left or right) zero semigroups as well as rectangular bands 
 are 
 abelian.  

 Since the multiplication of a left/right zero semigroup $S$ is essentially unary, we see
\[ M(1,1) = \left\{ \begin{bmatrix} a & a \\ b & b \end{bmatrix},\begin{bmatrix} c & d \\ c & d \end{bmatrix} \st a,b,c,d\in S\right\}. \]
 Clearly if the entries in the first row of any element in $M(1,1)$ are equal, then so are the entries in the second.
 Hence $C(1,1,0)$ and $[1,1]=0$ by~\eqref{eq:tc}.

 Further, any rectangular band is a direct product of a left zero and a right zero semigroup, hence abelian as well.
\end{exa}

 In general $\A/(1]^{d+1}$ is the largest left $d$-nilpotent quotient of $\A$, etc. However a quotient of an
 abelian (left/right nilpotent) algebra may not be abelian (left/right nilpotent) as the following shows.
 
\begin{exa} \label{ex:sl}
  The $2$-element semilattice $(\{0,1\},\cdot)$ is a commutative monoid but not abelian (it is a \emph{snag})
  as witnessed by
\[ \begin{bmatrix} 0 & 0 \\ 1 & 1 \end{bmatrix}
  \begin{bmatrix} 0 & 1 \\ 0 & 1 \end{bmatrix} =
  \begin{bmatrix} 0 & 0 \\ 0 & 1 \end{bmatrix} \in M(1,1), \]
 further not 
 solvable since $[1,1]
 =1$ by~\eqref{eq:tc} and not $d$-supernilpotent for any $d\in\N$ by~\eqref{eq:htc} since e.g. for $d=2$
\[ \begin{bmatrix} 0 & 0 \\ 0 & 0 \\1 & 1 \\ 1 & 1 \end{bmatrix}
  \begin{bmatrix} 0 & 0 \\1 & 1 \\ 0 & 0 \\ 1 & 1 \end{bmatrix}
  \begin{bmatrix} 0 & 1 \\ 0 & 1 \\ 0 & 1 \\ 0 & 1 \end{bmatrix} =
\begin{bmatrix} 0 & 0 \\ 0 & 0 \\ 0 & 0 \\ 0 & 1 \end{bmatrix} \in M(1,1,1). \]
 Note that this example also shows that abelianess, 
solvability, left and right nilpotence, 
and supernilpotence
 of semigroups are not preserved by adjoining $0$ or $1$.
  
 The natural numbers $\N$ including $0$ form a commutative and cancellative monoid under addition,
 hence an abelian monoid (see Theorem~\ref{thm:nmonoid}).
 Let $I := \N\setminus\{0\}$. Then $\N/I = \{0,I\}$ is a $2$-element semilattice,
 hence non-abelian by Example~\ref{ex:sl}.
\end{exa}

\subsection{Results}
 Throughout the paper `nilpotent semigroup' will refer to (left or right) nilpotent in the sense of commutator theory
 that $(1]^{d+1} = 0$ or $[1)^{d+1}=0$ for some $d\in\N$. When we mean `nilpotent' in the classical semigroup sense that
 \[ S^{d+1}:= \{a_1\dots a_{d+1} \st a_1,\dots,a_{d+1}\in S \} = 0, \]
 we will mention this explicitly. 

 In~\cite[Propositions 3.5, 3.6]{RM:HCS} Mudrinski and Radovi{\'c} showed that for a semigroup with $0$ the conditions
 of solvability, nilpotence (in the sense of commutator theory that $(1]^{d+1} = 0$), supernilpotence and nilpotence
 (in the classical semigroup theory sense that $S^{d+1} = 0$) are all equivalent.
 In~\cite[Proposition 5.11]{RM:SS} they proved that completely simple semigroups are solvable (nilpotent)
 iff their subgroups are solvable (nilpotent).
 In~\cite[Proposition 4.4]{RM:SS} they showed that nilpotent and supernilpotent orthodox semigroups are the same.
 
 In the first main result of this paper we unify these cases by characterizing semigroups with a completely simple
 ideal that have one of these properties. Recall that a semigroup $S$ is a \emph{nilpotent extension} of an ideal $K$ if
\[ S^m \subseteq K \text{ for some } m\in\N. \]
 
\begin{thm} \label{thm:csi}
 Let $S$ be a semigroup  with a completely simple ideal $K$.
\begin{enumerate}
\item \label{it:solvable}
 Then $S$ is solvable iff $S$ is a nilpotent extension of $K$ and all subgroups of $K$ are solvable. 
\item \label{it:nilpotent}
 Then $S$ is (left or right) nilpotent iff $S$ is supernilpotent iff $S$ is a nilpotent extension of $K$
 and all subgroups of $K$ are nilpotent.
\end{enumerate}  
\end{thm}

 Theorem~\ref{thm:csi} applies in particular to all finite semigroups. It is proved in Section~\ref{sec:csi}.
 For~\eqref{it:nilpotent} we establish the same structural description for both nilpotent and supernilpotent semigroups.
 We do not know of a direct proof of the equivalence of nilpotence and supernilpotence, e.g.,
 by expressing higher commutators via some iteration of binary commutators in this case.

 A semigroup $S$ is \emph{eventually regular} if every element $a\in S$ has some power which is regular, i.e.
 some $n\geq 1$ and $b\in S$ such that $a^n ba^n = a^n$~\cite{Ed:ERS}.
 In particular finite, periodic, group-bound and regular semigroups are eventually regular.

 Recall from Example~\ref{ex:sl} that the idempotents $E(S)$ of any solvable, nilpotent or supernilpotent semigroup
 $S$ form an antichain under the usual order
\[ e \leq f \text{ if } ef = fe = e, \]
 i.e. all idempotents are \emph{primitive}.
 A straightforward adaptation of the proof of~\cite[Theorem 3.3.3 (4)$\Rightarrow$(1)]{Ho:FST} shows that any eventually
 regular semigroup with a primitive idempotent has a unique completely simple ideal (see Lemma~\ref{lem:er}).
 Hence Theorem~\ref{thm:csi} immediately yields the following.
 
\begin{cor} \label{cor:er}
 Let $S$ be an eventually regular semigroup.  
\begin{enumerate}
\item
 Then $S$ is solvable iff $S$ is a nilpotent extension of a completely simple semigroup with solvable subgroups.
\item
 Then $S$ is (left or right) nilpotent iff $S$ is supernilpotent iff $S$ is a nilpotent extension of
 a completely simple semigroup with nilpotent subgroups. 
\end{enumerate}
\end{cor}

 In particular every solvable, nilpotent supernilpotent, respectively, regular semigroup is clearly a completely
 simple semigroup with the same property~\cite[cf. Proposition 5.11]{RM:CCS}. 
 Also every solvable, nilpotent supernilpotent, respectively, inverse semigroup is a group with the same
 property~\cite[Theorem 1.3]{KS:ACI}. 

 Finally we give a characterization of nilpotent monoids.

\begin{thm} \label{thm:nmonoid}
 For a monoid $S$ and $n\in\N$ the following are equivalent:
\begin{enumerate}
\item \label{it:crn}
 $S$ is cancellative and right $n$-nilpotent.
\item \label{it:rln}
 $S$ is right and left $n$-nilpotent.
\item \label{it:eng}
 $S$ embeds into an $n$-nilpotent group.  
\end{enumerate}
\end{thm}

 For $n=1$ Theorem~\ref{thm:nmonoid} simplifies to: a monoid is cancellative and commutative iff it is abelian iff
 it embeds into an abelian group.
 Theorem~\ref{thm:nmonoid} is proved in Section~\ref{sec:monoid}.
 There we also show that for monoids, $2$-supernilpotence is equivalent to left and right $2$-nilpotence in
 Theorem~\ref{thm:2sn}. The following remain open.

\begin{que}   \label{que:sn}
 Does supernilpotence imply left and right nilpotence for all semigroups? For all monoids? What about the converse?
\end{que}

\section{Commutators of Rees congruences and on reducts}  \label{sec:prelims}

 We collect a few auxiliary results necessary for proving our main theorems.

 For an ideal $I$ of $S$ let
 \[ \rho_I := I\times I\ \cup\ \{ (s,s) \st s\in S\} \]
 denote the \emph{Rees congruence} induced by $I$. Then $S/\rho_I = S/I$.

 We recall an upper bound for the commutator of Rees congruences due to Mudrinski and Radovi{\' c}.
 
\begin{lem}~\cite[Lemma 2.8, Theorem 1.1]{RM:HCS} \label{lem:ideals}
 Let $S$ be a semigroup with ideals $I_1,\dots,I_n$ for $n\in\N$, let $J := \prod_{f\in S_n} I_{f(1)}\dots I_{f(n)}$.
 Then
 \[ [\rho_{I_1},\dots,\rho_{I_n}] \leq \rho_J \]
 with equality if $S$ has $0$.
\end{lem}

 The following reduces commutators of congruences below $\rho_I$ to commutators in the ideal $I$ of $S$.

\begin{lem} \label{lem:MIab}
 Let $S$ be a semigroup with an ideal $I$ and congruences $\alpha_1,\dots\alpha_n\leq\rho_I$.
 Then $[\alpha_1,\dots\alpha_n]_S$ is the congruence of $S$ generated by $[\alpha_1|_I,\dots\alpha_n|_I]_I$.
\end{lem}  
 
\begin{proof}
 The result will follow from
\begin{equation} \label{eq:MIab}
 M_S(\alpha_1,\dots\alpha_n) = M_I(\alpha_1|_I,\dots\alpha_n|_I) \cup \{ \begin{bmatrix} s & s \\ \vdots &\vdots \\  s & s \end{bmatrix} \st s\in S\setminus I \}.
\end{equation} 
 We denote the set on the right hand side of~\eqref{eq:MIab} by $V$. Clearly $V\subseteq M_S(\alpha_1,\dots\alpha_n)$.
 For the converse inclusion, note that every generator of $M_S(\alpha_1,\dots\alpha_n)$ is contained in $V$.
 Further $V$ is closed under multiplication since the set of generators of $M_I(\alpha_1|_I,\dots\alpha_n|_I)$
 is closed under multiplication with elements $\begin{bmatrix} s & s \\ \vdots &\vdots \\  s & s \end{bmatrix}$
 for $s\in S$.
 Hence $M_S(\alpha_1,\dots\alpha_n) \subseteq V$ and~\eqref{eq:MIab} is proved.

 For a congruence $\delta$ of $I$, let $\overline{\delta}$ denote the congruence of $S$ generated by $\delta$. 
 By~\eqref{eq:MIab} we have $C(\alpha_1|_I,\dots\alpha_n|_I,\delta)$ iff $C(\alpha_1,\dots\alpha_n,\overline{\delta})$.
 Since $[\alpha_1|_I,\dots\alpha_n|_I]_I$ is the smallest $\delta$ such that $C(\alpha_1|_I,\dots\alpha_n|_I,\delta)$,
 it follows that $\overline{[\alpha_1|_I,\dots\alpha_n|_I]_I}$ is the smallest $\gamma$ such that
 $C(\alpha_1,\dots\alpha_n,\gamma)$. Hence
\[ [\alpha_1,\dots\alpha_n]_S = \overline{[\alpha_1|_I,\dots\alpha_n|_I]_I}. \]
\end{proof}

 Let $\A$ be an algebra with congruences $\alpha_1,\dots,\alpha_n$, and let $\A_0$ be a reduct of $\A$.
 Then $\alpha_1,\dots,\alpha_n$ are clearly also congruences of $\A_0$,
\[ M_{\A_0}(\alpha_1,\dots,\alpha_n) \subseteq M_\A(\alpha_1,\dots,\alpha_n) \]  
 and consequently
\[ [\alpha_1,\dots,\alpha_n]_{\A_0} \subseteq  [\alpha_1,\dots,\alpha_n]_\A. \]
 The converse inclusion clearly does not hold in general, e.g., the group $(\mathbb{Z}_2,+)$ is abelian but the ring
 $(\mathbb{Z}_2,+,\cdot)$ is not by Example~\ref{ex:sl}.
 Still we have the following convenient equivalence between congruences and their congruences of a group and
 its semigroup reduct.
 
\begin{lem}  \label{lem:sgreduct}
 Let $n\in\N$ and $\alpha_1,\dots,\alpha_n$ be congruences of a group $(G,\cdot,^{-1},1)$. Then
\[ [\alpha_1,\dots,\alpha_n]_{(G,\cdot,^{-1},1)} = [\alpha_1,\dots,\alpha_n]_{(G,\cdot)}. \]
\end{lem}

\begin{proof}
 It is well known and not hard to show that the group $(G,\cdot,^{-1},1)$ and its semigroup reduct $(G,\cdot)$
 have exactly the same congruences. Further $M_{(G,\cdot,^{-1},1)}(\alpha_1,\dots,\alpha_n)$ and
 $M_{(G,\cdot)}(\alpha_1,\dots,\alpha_n)$ have the same universe. 
 So the elements of one satisfy the term condition~\eqref{eq:htc} for $\delta$ iff the elements of the other
 do. Hence $[\alpha_1,\dots,\alpha_n]_{(G,\cdot,^{-1},1)} = \delta$ iff $[\alpha_1,\dots,\alpha_n]_{(G,\cdot)} = \delta$.
\end{proof}

\section{Extensions of completely simple ideals}  \label{sec:csi}

 This section consists entirely of the proof of Theorem~\ref{thm:csi}.
 So let $S$ denote a semigroup with completely simple ideal $K$ throughout.
 Then $K$ is the smallest ideal of $S$, hence the kernel of the semigroup (potentially $K=0$).

 We first establish the forward implications of Theorem~\ref{thm:csi} as well as some connections between the
 solvability or (super)nilpotence class of $S$ and its subgroups.
 
\begin{lem} \label{lem:csi}
 Let $d\in\N$ and $S$ a semigroup with a completely simple ideal $K$.
\begin{enumerate}  
\item \label{it:scsi}
 If $S$ is $d$-solvable, then $S^{2^d} = K$ and all subgroups of $K$ are $d$-solvable.
\item \label{it:ncsi}
 If $S$ is left $d$-nilpotent or right $d$-nilpotent or $d$-supernilpotent,
 then $S^{d+1} = K$ and all subgroups of $K$ are $d$-nilpotent.
\end{enumerate}  
\end{lem}

\begin{proof}
 Assume that $S$ is $d$-solvable, left/right $d$-nilpotent, $d$-supernilpotent, respectively.
 Then all subsemigroups of $S$ have the same property. By Lemma~\ref{lem:sgreduct} also the subgroups of $K$ in
 the extended group signature $\cdot,^{-1},1$ have the same property. Recall that for a group left $d$-nilpotent,
 right $d$-nilpotent, $d$-supernilpotent and $d$-nilpotent the classical group theory sense are all equivalent.
 Hence it only remains to show that $S^m = K$ for $m\in\N$ as claimed.

\eqref{it:scsi}
 Assume $S$ is $d$-solvable.  
 We show $S^{2^d} \subseteq K$ by induction on the solvable class $d$ of $S$.
 The base case $d=0$ just means that the congruences $1$ and $0$ of $S$ are equal, i.e., $S^1=S=0$ is trivial.

 Consider $d>0$ and $\alpha := [1]^{d-1} \neq 0$, $[\alpha,\alpha]=0$ in the following.
 Then $S/\alpha$ is $d-1$-solvable with $K/\alpha$ a completely simple ideal.
 By induction assumption
\[ S^{2^{d-1}} \subseteq \bigcup_{a\in K} a/\alpha. \] 
 We show that for all $a,b\in K$ and $x,y\in S$ such that $x\, \alpha\, a, y\, \alpha\, b$ we have 
\[ xy\in K. \]
 Since $K$ is completely simple, there exists $e\in E(K)$ such that $ae=a$ and $eb=b$. The congruences
\[ x\, \alpha\, a = ae\, \alpha\, xe,  \quad y\, \alpha\, b = eb\, \alpha\, ey \]
 yield $x\, \alpha\, xe$ and $y\, \alpha\, ey$.
 Hence
 \[ \begin{bmatrix} xe & xe \\  x & x \end{bmatrix} \cdot\begin{bmatrix} y & ey \\  y & ey \end{bmatrix} =
   \begin{bmatrix} xey & xey \\  xy & xey \end{bmatrix} \in M_S(\alpha,\alpha). \]
 Since $[\alpha,\alpha]=0$ by assumption, by~\eqref{eq:tc} equality in the first row of the above matrix yields equality
 in the second, that is, $xy = xey\in K$. Thus $S^{2^d} \subseteq K$ as claimed.

\eqref{it:ncsi}
{\bf Case, $S$ is left $d$-nilpotent:}
 We show $S^{d+1} \subseteq K$ by induction on the nilpotence class $d$ of $S$.
 The base case $d=0$ just means that the congruences $1$ and $0$ of $S$ are equal, i.e., $S^1=S=0$ is trivial.

 Consider $d>0$ and $\alpha := (1]^d \neq 0$, $[1,\alpha]=0$ in the following.
 Then $S/\alpha$ is left $(d-1)$-nilpotent with $K/\alpha$ a completely simple ideal.
 By induction assumption
\[ S^{d} \subseteq \bigcup_{a\in K} a/\alpha. \] 
 We show that for all $a\in K$ and $x,y\in S$ such that $x\, \alpha\, a$ we have 
\[ xy\in K. \]
 Since $K$ is completely simple, there exists $e\in E(K)$ such that $ae=a$. Then
\[ x\, \alpha\, a = ae\, \alpha\, xe. \]
 Hence
 \[ \begin{bmatrix} x & xe \\  x & xe \end{bmatrix} \cdot\begin{bmatrix} e & e \\  y & y \end{bmatrix} =
   \begin{bmatrix} xe & xe \\  xy & xey \end{bmatrix} \in M_S(1,\alpha). \]
 Since $[1,\alpha]=0$ by assumption, by~\eqref{eq:tc} equality in the first row of the above matrix yields equality in the second,
 that is, $xy = xey\in K$. Thus $S^{d+1} \subseteq K$ as claimed.

 {\bf Case, $S$ is right $d$-nilpotent:} This is very similar to the previous case changing only
 $\alpha := [1)^d \neq 0$ and considering
 \[ \begin{bmatrix} xe & xe \\  x & x \end{bmatrix} \cdot\begin{bmatrix} y & ey \\  y & ey \end{bmatrix} =
   \begin{bmatrix} xey & xey \\  xy & xey \end{bmatrix} \in M_S(\alpha,1). \]
 Then the assumption $[\alpha,1]=0$ yields $xy = xey\in K$ by~\eqref{eq:tc} and $S^{d+1} \subseteq K$ again.
 
{\bf Case, $S$ is $d$-supernilpotent:} Let $a_1,\dots,a_{d+1}\in S$, $b\in K$ and $e\in E(K)$ such that
\begin{equation} \label{eq:eab}
 ea_1\dots a_db = a_1\dots a_db.
\end{equation}  
 Note that $e$ can be chosen as the identity of the subgroup of $K$ containing $a_1\dots a_db\in K$.
 Consider the product of the following $d+1$ elements in $M_S(1,\dots,1) \leq S^{2^d \times 2}$, where for $1<i\leq d+1$
 the $i$-th factor has half of its entries $a_i\dots a_db$ and half $a_i$:
 \[  \begin{tiny} \left(\begin{array}{cc}
    a_1 & ea_1 \\
    \vdots & \vdots \\ 
    \vdots & \vdots \\ 
    \vdots & \vdots \\ 
    \vdots & \vdots \\ 
    \vdots & \vdots \\ 
    \vdots & \vdots \\ 
    \vdots & \vdots \\ 
    \vdots & \vdots \\ 
    \vdots & \vdots \\ 
    \vdots & \vdots \\ 
    \vdots & \vdots \\ 
    a_1 & ea_1 \\  
    \end{array}\right)  \cdot
  \left(\begin{array}{cc}
    a_2\dots a_d b &  a_2\dots a_d b \\
    \vdots & \vdots \\ 
    \vdots & \vdots \\ 
    \vdots & \vdots \\ 
    \vdots & \vdots \\ 
    \vdots & \vdots \\ 
     a_2\dots a_d b &  a_2\dots a_d b \\
    a_2 & a_2 \\  
    \vdots & \vdots \\ 
    \vdots & \vdots \\ 
    \vdots & \vdots \\ 
    \vdots & \vdots \\ 
    \vdots & \vdots \\ 
    a_2 & a_2 \\  
    \end{array}\right) \cdot
 \left(\begin{array}{cc}
    a_3\dots a_db & a_3\dots a_db \\
    \vdots & \vdots \\ 
    \vdots & \vdots \\ 
    a_3\dots a_db & a_3\dots a_db \\
    a_3 & a_3 \\  
    \vdots & \vdots \\ 
    \vdots & \vdots \\ 
    a_3 & a_3 \\  
     a_3\dots a_db & a_3\dots a_db \\
    \vdots & \vdots \\ 
    \vdots & \vdots \\ 
    a_3\dots a_db & a_3\dots a_db \\
    a_3 & a_3 \\  
    \vdots & \vdots \\ 
    \vdots & \vdots \\ 
    a_3 & a_3 \\  
   \end{array}\right)\ \dots\
\left(\begin{array}{cc}
    b & b \\
    a_{d+1} & a_{d+1} \\
    b & b \\
    a_{d+1} & a_{d+1} \\
    \vdots & \vdots \\ 
    \vdots & \vdots \\ 
    \vdots & \vdots \\ 
    \vdots & \vdots \\ 
    \vdots & \vdots \\ 
    \vdots & \vdots \\ 
    \vdots & \vdots \\ 
    \vdots & \vdots \\ 
    \vdots & \vdots \\ 
    b & b \\
    a_{d+1} & a_{d+1} \\
   \end{array}\right)\end{tiny}
\]
 A straightforward inductive argument using~\eqref{eq:eab} shows that for any $i>1$,
 the product of the first $i$ factors above has equality in all rows except for the bottom $2^{d+1-i}$ rows,
 which are all of the form $(a_1\dots a_i, ea_1\dots a_i)$.
 In particular, for the full product of length $i=d+1$ the two entries in every row are equal but for the last one,
 which is $(a_1\dots a_{d+1}, ea_1\dots a_{d+1})$.
 Hence the definition of the  $(d+1)$-ary higher commutator~\eqref{eq:htc} yields that
\[ (a_1\dots a_{d+1}, ea_1\dots a_{d+1}) \in  [\underbrace{1,\dots,1}_{d+1}]. \]
Since $S$ is $d$-supernilpotent, this means that
\[ a_1\dots a_{d+1} = ea_1\dots a_{d+1}\in K. \]
Thus $S^{d+1}\subseteq K$ as claimed.
\end{proof}

 For the proofs of the backward implications in Theorem~\ref{thm:csi} we need a few more auxiliary results.
 First we consider the case that all subgroups of $S$ are trivial.
 
\begin{lem} \label{lem:rbi}
 Let $S$ be a semigroup and $m\in\N$ such that $S^m$ is a rectangular band. Then
\begin{enumerate}
\item \label{it:arbi}   
 $S$ is abelian if $m=1$.  
\item \label{it:srbi}
 $S$ is solvable of class at most $\lceil\log_2(m)\rceil+1$ (class at most $\lceil\log_2(m)\rceil$ if $S^m=0$).   
\item \label{it:nrbi}
 $S$ is left and right nilpotent as well as supernilpotent of class at most $2m$ (class at most $m$ if $S^m=0$).
\end{enumerate}  
\end{lem}

\begin{proof}
 Let $K := S^m$ be a rectangular band.

\eqref{it:arbi} follows from Example~\ref{exa:rectangular}.
 
\eqref{it:srbi} Let $\ell := \lceil\log_2(m)\rceil$. We claim that the series of Rees congruences
\[ 1_S = \rho_S \geq \rho_{S^2} \geq \rho_{S^4} \geq \dots \rho_{S^{2^\ell}} = \rho_{K} \geq 0_S \]
 has abelian sections. For that note that $[\rho_{S^i},\rho_{S^i}] \leq \rho_{S^{2i}}$ for all $i\geq 1$ by 
 Lemma~\ref{lem:ideals}.
 If $K=0$, then $\rho_K$ is already trivial and $S$ is solvable of class at most $\lceil\log_2(m)\rceil$.
 In any case, since $[1_K,1_K]_K = 0_K$ by~\eqref{it:arbi}, Lemma~\ref{lem:MIab} yields $[\rho_K,\rho_K] = 0_S$.
 Hence $S$ is solvable of class  at most $\ell+1$.

 \eqref{it:nrbi}
 {\bf Claim $S$ is left and right nilpotent:}
 Since rectangular bands are direct product of left zero and right zero semigroups,
 $S$ has congruences $\lambda,\rho \leq \rho_K$ such that $K\cong K/\lambda|_K \times K/\rho|_K$ for left zero
 $K/\lambda|_K$ and right zero $K/\rho|_K$.
 Define $\alpha_0 := \rho_K$ and 
\[ \alpha_i := \Cg_S( (sa,sb) \st s\in S^i, a,b\in K ) \vee \lambda \]
 the congruence of $S$ generated by the pairs $(sa,sb)$ for $s\in S^i, a,b\in K$, $i\in [m]$ together with $\lambda$.
 Note that $\alpha_m=\lambda$ since $S^m = K$.

 We claim that
\begin{equation} \label{eq:1ai}
 [1_S,\alpha_i],\, [\alpha_i,1_S] \leq\alpha_{i+1} \text{ for all } i\leq m.
\end{equation}
 First, since $K/\lambda$ is left zero, every non-generator in $M_S(1_s,\alpha_i)$ is $\lambda$-congruent to
\[ \begin{bmatrix} a & a \\ b & b \end{bmatrix} \begin{bmatrix} c & d \\ c & d \end{bmatrix} \equiv_\lambda \begin{bmatrix} ac & ad \\ bc & bd \end{bmatrix} \]
 for some $a,b\in S, (c,d)\in\alpha_i$. By the definition of $\alpha_i$, we have $n\in\N, s_1,\dots,s_n\in S^i$
 and $x_1,\dots,x_{n+1}\in K$ such that
\begin{equation} \label{eq:csx}
 c  = s_1x_1,\ s_1x_2 \equiv_\lambda s_2x_3,\ \dots, s_{n-1}x_{n-1} \equiv_\lambda s_nx_n,\ s_nx_{n+1} = d.
\end{equation}
 Multiplying this chain on the left by $a, b$, respectively, we see that $(ac,ad),(bc,bd)\in\alpha_{i+1}$.
 Hence 
 $[1_S,\alpha_i]\leq\alpha_{i+1}$ by~\eqref{eq:tc}.

 Second, every non-generator in $M_S(\alpha_i,1_S)$ is $\lambda$-congruent to
\[ \begin{bmatrix} a & b \\ a & b \end{bmatrix} \begin{bmatrix} c & c \\ d & d \end{bmatrix} \equiv_\lambda \begin{bmatrix} ac & bc \\ ad & bd \end{bmatrix}  \]
 for some $a,b\in S, (c,d)\in\alpha_i$. In particular we have a chain as in~\eqref{eq:csx}.
 Multiplying this on the left by $a,b$, respectively, we see that $(ac,ad),(bc,bd)\in\alpha_{i+1}$.
 
 Now assume that $(ac,bc) \in\alpha_{i+1}$. Then  $(ad,bd) \in\alpha_{i+1}$ follows and we have
 $[\alpha_i,1_S]\leq\alpha_{i+1}$ by~\eqref{eq:tc}. Thus~\eqref{eq:1ai} is proved.

 Dually for $\beta_0 := \rho_K$,
\[ \beta_i := \Cg_S( (as,bs) \st s\in S^i, a,b\in K ) \vee \rho \]
 for $i\in [m]$ and $\beta_m = \rho$, we have
\begin{equation} \label{eq:1bi}
 [1_S,\beta_i],\, [\beta_i,1_S] \leq\beta_{i+1} \text{ for all } i\leq m.
\end{equation}

 We claim that 
\[ 1_S =\rho_S \geq \rho_{S^2} \geq \dots \rho_{S^m} = \rho_K \geq \alpha_1\wedge\beta_1 \geq \dots \geq \alpha_m\wedge\beta_m = \lambda \wedge \rho = 0_S \] 
 is a (left and right) central series.
 For that note that $[1_S,\rho_{S^i}]=[\rho_{S^i},1_S] \leq \rho_{S^{i+1}}$ for all $i\geq 1$ by 
 Lemma~\ref{lem:ideals}.
 If $K=0$, then $\rho_K$ is already trivial and $S$ is nilpotent of class at most $m$.
 In any case,
\[ [1_S,\alpha_i\wedge\beta_i],\, [\beta_i\wedge\alpha_i,1_S]\leq\alpha_{i+1}\wedge\beta_{i+1} \]
 by~\eqref{eq:1ai}, ~\eqref{eq:1bi} and the monotonicity of the commutator.
 Hence $S$ is left and right nilpotent of class  at most $2m$.

 {\bf Claim $S$ is supernilpotent:}
  Since $S^m=K$ is a rectangular band, $S$ satisfies
\[ x_1\dots x_myz_1\dots z_m = x_1\dots x_mz_1\dots z_m \]
 for all $ x_1,\dots, x_m,y,z_1,\dots,z_m \in S$. Hence all term operations of $S$ have essential arity
 at most $2m$. It follows from~\eqref{eq:htc} that every higher commutator of $2m+1$ congruences is trivial.
 In particular $S$ is $2m$-supernilpotent. 
\end{proof}

 In the general case $S$ has a completely simple ideal $K$.
 Then $K$ is isomorphic to a \emph{Rees matrix semigroup} $M(G;I,\Lambda;P)$ for a group $G$, sets $I,\Lambda$ and
 $P\in G^{\Lambda\times I}$ by an isomorphism
\[ K \to I\times G \times\lambda,\ x\mapsto (x_I,x_G,x_\lambda). \]
 Recall that the multiplication in $M(G;I,\Lambda;P)$ is defined by
\[ (i,g,\lambda)(j,h,\mu) := (i,g P_{\lambda j} h,\mu) \quad \text{ for }  (i,g,\lambda),(j,h,\mu)\in I\times G\times\Lambda. \]
For a normal subgroup $N$ of $G$
 \[ \sigma_N := \{ (a,b)\in K^2 \st a_I = b_I, a_\Lambda = b_\Lambda, a_G \equiv_N b_G \} \cup \{(s,s)\st s\in S\} \]
 is a congruence of $S$ contained in the Rees congruence $\rho_K$.
 To see that let $(a,a'), (b,b') \in\sigma_N$ and consider the following cases:
\begin{enumerate} 
\item $a,b\in S\setminus K$: Then $a=a'$ and $b=b'$ implies $ab\ \sigma_N\ bb'$.
\item \label{it:abK}
 $a,b\in K$: From the definition of the Rees matrix semigroup, it is straightforward that
 $ab\ \sigma_N\ a'b'$.
\item $a\in K, b\in S\setminus K$: Let $e$ be the identity in the group $\mathcal{H}$-class $H_a=H_{a'}$ of $S$.
 Since $eb=eb'$ is in $K$, case~\eqref{it:abK} yields 
 $$ab = aeb = a(eb')\ \sigma_N\ a'(eb') = a'b'.$$
\item  $a\in S\setminus K, b \in K$: This case is dual to the previous one. 
\end{enumerate}
 This completes the proof that $\sigma_N$ is a congruence on $S$. Clearly $\sigma_N \leq \rho_K$.

 We reduce the commutators of congruences of the form above to classical commutators $[M,N]$ and $[G,N]$
 of normal subgroups of $G$.
 
\begin{lem} \label{lem:sN}
 Let $S$ be a semigroup with a completely simple ideal $K \cong M(G;I,\Lambda;P)$
 for a group $G$, sets $I,\Lambda$ and $P\in G^{\Lambda\times I}$. Let $M,N$ be normal subgroups of $G$.
 Then 
\begin{enumerate}
\item \label{it:sMsN}
 $[\sigma_M,\sigma_N] = \sigma_{[M,N]}$,
\item \label{it:1SsN}
 $[1_S,\sigma_N] = [\sigma_N,1_S] = \sigma_{[G,N]}.$
\end{enumerate}  
\end{lem}

\begin{proof}
\eqref{it:sMsN} By Lemma~\ref{lem:MIab} $[\sigma_M,\sigma_N]_S$ is generated by $[\sigma_M|_K,\sigma_N|_K]_K$.
 Since the latter is $\sigma_{[M,N]}|_K$ by~\cite[Theorem 1.1]{RM:CCS}, the result follows.

\eqref{it:1SsN} First we show
\begin{equation} \label{eq:MK1s}
 M_S(1_S,\sigma_N) = M_K(1_K,\sigma_N|_K) \cup \{ \begin{bmatrix} a & a \\ b & b \end{bmatrix} \st a,b\in S \}.
\end{equation} 
 We denote the set on the right hand side of~\eqref{eq:MK1s} by $V$. Clearly $V\subseteq M_S(1_S,\sigma_N)$.
 For the converse inclusion, note that every generator of $M_S(1_S,\sigma_N)$ is contained in $V$.
 Further $V$ is closed under multiplication since the product of a generator of $M_K(1_K,\sigma_N|_K)$ with an element
 $\begin{bmatrix} a & a \\  b & b \end{bmatrix}$ for $a,b\in S$ is always in $M_K(1_K,\sigma_N|_K)$:
 For $(c,d)\in\sigma_N \cap (K\times K)$, let $e$ be the unique idempotent in the group $\mathcal{H}$-class $H_c=H_d$ of S.
 Then
 \[ \begin{bmatrix} a & a \\  b & b \end{bmatrix} \cdot\begin{bmatrix} c & d \\  c & d \end{bmatrix} =
   \underbrace{\begin{bmatrix} ae & ae \\  be & be \end{bmatrix}}_{\in K^{2\times 2}} \cdot\begin{bmatrix} c & d \\  c & d \end{bmatrix} \in M_K(1_K,\sigma_N|_K), \]
 \[ \begin{bmatrix} c & d \\  c & d \end{bmatrix} \cdot \begin{bmatrix} a & a \\  b & b \end{bmatrix}=
    \begin{bmatrix} c & d \\  c & d \end{bmatrix} \cdot \underbrace{\begin{bmatrix} ea & ea \\  eb & eb \end{bmatrix}}_{\in K^{2\times 2}} \in M_K(1_K,\sigma_N|_K). \]
 The remaining cases are trivial.
 So $M_S(1_S,\sigma_N) \subseteq V$ and~\eqref{eq:MK1s} is proved.

 By~\cite[Theorem 1.1]{RM:CCS} we have $[1_K,\sigma_N|_K] = \sigma_{[G,N]}|_K$. 
 Together with~\eqref{eq:MK1s} and~\eqref{eq:tc} this implies 
 $[1_S,\sigma_N] \leq \sigma_{[G,N]}$.
 The converse inclusion is clear and so we have equality.

 The proof that $[\sigma_N,1_S] = \sigma_{[N,G]} = \sigma_{[G,N]}$ is symmetric. 
\end{proof}

 Now we are ready to show the backward implications of Theorem~\ref{thm:csi} together with some
 connections between solvability class, nilpotence class, respectively, of $S$ and of its subgroups.
 Note that for $n=0$, all subgroups of $K$ are trivial and the following specializes to Lemma~\ref{lem:rbi}.

\begin{lem} \label{lem:csib}
 Let $m,n\in\N$ and $S$ a semigroup with a completely simple ideal $K$ such that $S^m = K$.
\begin{enumerate}  
\item \label{it:scsib}
 If all subgroups of $K$ are $n$-solvable, then $S$ is solvable of class at most $\lceil \log_2(m)\rceil+n+1$.
\item \label{it:ncsib}
 If all subgroups of $K$ are $n$-nilpotent, then $S$ is left and right nilpotent of class at most $2m+n$
 as well as supernilpotent of class at most $2m\cdot\max(n,1)$.
\end{enumerate}  
\end{lem}

\begin{proof}
  Let $K$ be isomorphic to the Rees matrix semigroup $M(G;I,\Lambda;P)$ for some group $G$, sets $I,\Lambda$ and
  $P\in G^{\Lambda \times I}$.

 \eqref{it:scsib}
 Let $\ell := \lceil \log_2(m)\rceil$. By Lemma~\ref{lem:rbi}\eqref{it:srbi} $S/\sigma_G$ is solvable of class at most $\ell+1$.
 Assume $G$ is solvable of class $n$, that is, its derived series is
 \[ G = G^{(0)} \geq G^{(1)}  \geq \dots \geq G^{(n)} = 1, \]
 where $G^{(i+1)} := [G^{(i)},G^{(i)}]$ is the classical group commutator for $i\in\N$. Then
\[ 1_S =\rho_S \geq \rho_{S^2} \geq \dots \geq \rho_{S^{2^\ell}} = \rho_K \geq \sigma_G \geq \sigma_{G^{(1)}} \geq \dots \geq \sigma_{G^{(n)}} = 0_S \]  
 is an abelian series of $S$ by Lemma~\ref{lem:sN}\eqref{it:sMsN}.
 So $S$ is solvable of class $\leq \ell+1+n$.

\eqref{it:ncsib}
 Assume $G$ is nilpotent of class $n$, that is, its lower central series is
\[ G = \gamma_1 G \geq \gamma_2 G  \geq \dots \geq \gamma_{n+1} G = 1, \]
 where $\gamma_{i+1} G := [\gamma_iG,G]$  is the classical group commutator for $i\geq 1$.

{\bf Claim $S$ is nilpotent:} Consecutive quotients in
\[ \sigma_G \geq \sigma_{\gamma_2 G} \geq \dots \geq \sigma_{\gamma_{n+1} G} = 0_S \]  
 are central in $S$ by Lemma~\ref{lem:sN}\eqref{it:1SsN}.
 By Lemma~\ref{lem:rbi}\eqref{it:nrbi} $S/\sigma_G$ is left and right nilpotent of class at most $2m$.
 So $S$ is left and right nilpotent of class $\leq 2m+n$.

 {\bf Claim $S$ is supernilpotent:} For $n=0$, the result is just Lemma~\ref{lem:rbi}\eqref{it:nrbi}.
 So we assume $n\geq 1$ in the following.
 The bulk of the proof consists of an analysis of elements in $M_S(1,\dots,1)$, for which we introduce a more versatile
 notation for its generators. For $\ell\leq k$, $1\leq i_1 < \dots < i_\ell \leq k$ and $f\colon [2]^\ell\to S$, define
\[ f^k_{\{i_1,\dots,i_\ell\}}\colon [2]^k\to S,\ (x_1,\dots,x_k) \mapsto f(x_{i_1},\dots,x_{i_\ell}). \]
 We note that $f^k_{\{i_1,\dots,i_\ell\}}$ is a $k$-ary \emph{minor} of $f$, which only depends on $\ell$ of its arguments.
 Since $k$ will be fixed throughout the proof below, we will also use the shorthand notation $f_A$ with
 $A:=\{i_1,\dots,i_\ell\}$.
 
 We can consider $M_S(1,\dots,1)$ as subsemigroup of $S^{[2]^k}$ that is generated by
\[ \{ f^k_{\{i\}} \st i\in [k], f\in S^{[2]} \}. \]
 Further the $k$-ary commutator $[1,\dots,1] = 0$ iff every $g\in M_S(1,\dots,1) \leq S^{[2]^k}$ satisfies
 the term condition (cf.~\eqref{eq:htc})
\begin{equation} \label{eq:TCk}
  \left(  g(x0)=g(x1) \text{ for all } x\in [2]^{k-1}\setminus \{(1,\dots,1)\} \right) \quad \Rightarrow \quad g(1,\dots,1,0)=g(1,\dots,1,1).
\end{equation}
 Here $x0, x1$ denote the $k$-tuples obtained by appending $0, 1$, respectively, to $x$.
 Note that the product of any $\ell< k$ generators of $M_S(1,\dots,1)$ is of the form $f^k_A$ for some
 $A\subseteq [k]$, $|A|=\ell$, and $f\in S^{[\ell]}$.
 Any such product of less than $k$ generators satisfies~\eqref{eq:TCk} trivially.

 For the remainder of the proof fix $k := 2mn+1$. For showing that $S$ is supernilpotent of class at most $k-1$,
 it clearly suffices to verify~\eqref{eq:TCk} for all elements $g$ of $M_S(1,\dots,1) \leq S^{[2]^{k}}$ that are
 products of $\ell\geq m$ generators of the form $f_{\{i\}}$, say
\[ g = a_1 \dots a_\ell. \] 
 Note that for any $j\in\{0,\dots,m-\ell\}$ the product $b_j := a_{j+1}\dots a_{j+m}$ is in $K$ by the assumption
 $S^m = K$.
 In particular, $b_j = g^{(j)}_{A_j}$ for some $A_j\subseteq [k]$ with $|A_j| = m$ and some $g^{(j)}\in K^{[2]^m}$.
 For $j\in [m-\ell]$, let $e_j\in K^{[2]^k}$ be a right identity of $b_j$. Then $e_j = h^{(j)}_{A_j}$ for some 
 $h^{(j)}\in E(K)^{[2]^m}$ and
\[ g =  a_1 \dots a_m   \prod_{j=1}^{m-\ell} (a_{m+j} e_j) =  g_{A_1} \prod_{j=1}^{m-\ell} (a_{m+j}h^{(j)}_{A_j}). \]
 Note that by the construction, for any $j\in [m-\ell]$, the product $a_{m+j}h^{(j)}_{A_j}$ is of the form
 $u^{(j)}_{A_j}$ for some $u^{(j)}\in K^{[2]^m}$.
 Thus $g$ is in the subsemigroup
\[ U := \langle u_A \st A\subseteq [k], |A| = m, u\in K^{[2]^m} \rangle \leq  K^{[2]^k}. \]
 In other words
\[ M_S(1,\dots,1) \leq U \cup \{ h_A \st A\subseteq [k], |A| = m-1, h\in S^{[2]^{m-1}} \}. \]
 As mentioned above it suffices to show that every $g\in U$ satisfies~\eqref{eq:TCk}.
 The projection of $U$ on its $I\times\Lambda$-components of $K$ yields a rectangular band,
 which is abelian and clearly satisfies~\eqref{eq:TCk}.
 So it only remains to consider the projection $V \leq G^{[2]^k}$ of $U$ on to the group $G$.
 Since the multiplication of two elements of $M(G;I,\Lambda;P)$ inserts elements from the sandwich matrix $P$,
 we obtain that $V$ is contained in the subgroup
\[ W := \langle h_A \st A\subseteq [k], |A| = 2m, h\in G^{[2]^{2m}} \rangle \leq G^{[2]^k}. \]
 The result follows once we have shown that every $g\in W$ satisfies~\eqref{eq:TCk}.

 For this purely group theoretic statement, define a \emph{difference operator} $\Delta$ that maps any $p$-ary
 function $u\colon [2]^p\to G$ to the $p-1$-ary function
 \[ \Delta u\colon [2]^{p-1}\to G,\ (x_1,\dots,x_p-1) \mapsto u(x_1,\dots,x_{p-1},0)^{-1}u(x_1,\dots,x_{p-1},1). \]
 We claim that $k$ applications of $\Delta$ to any $g\in W$ yields the constant $1$ function, that is,
\begin{equation} \label{eq:Dkg}
 \Delta^k g = 1  \text{ for all } g \in W.
\end{equation}
 Recall that the group $G$ (and hence $W$) is $n$-nilpotent. Hence by classical commutator collection in groups,
 every $g\in W$ can be written in a normal form using iterated classical group commutators of length at most $n$ of
 generators $h_A$ of $W$. Such an iterated commutator of $i\leq n$ generators
 with $|A_1| = \dots = |A_i| = 2m$ is of the form 
\[ [h^{(1)}_{A_1},\dots,h^{(i)}_{A_i}] = w_C \text{ for some } A_1\cup\dots\cup A_i\subseteq C\subseteq [k], |C|= 2mi, w\in (\gamma_i G)^{[2]^{2mi}}. \]
 Fixing a linear order on the subsets $B$ of $[k]$, we can then write any $g\in W$ in the form 
\begin{equation} \label{eq:gnf}
 g = \prod_{B\subseteq [k], |B| = 2mi, i\in [n]} w^{(B)}_B
\end{equation}
 for some $w^{(B)}\in (\gamma_i G)^{[2]^{2mi}}$ if $|B| = 2mi$.

 The particular order we want to use in~\eqref{eq:gnf} is reverse-lexicographical on the characteristic functions of
 $B\subseteq [k]$ with $|B| \leq 2mn = k-1$.
 Hence first come all the subsets $B$ with $k\not\in B$, then all the subsets $B$ with $k\in B$ but $k-1 \not\in B$,
 \dots and the last subset is $\{2,\dots,k\}$.

 For $g\in W$ as in~\eqref{eq:gnf}, we then have
\[ \Delta g = \Delta\left( \prod_{B\subseteq [k], |B| = 2mi, i\in [n]} w^{(B)}_B \right) = \prod_{\{k\}\subseteq B\subseteq k, |B| = 2mi, i\in [n]}  w^{(B)}_B \]
 since all the factors $w^{(B)}_B$ with $k\not\in B$ get cancelled in
$$(\Delta g)(x_1,\dots,x_{k-1}) = g(x_1,\dots,x_{k-1},0)^{-1}g(x_1,\dots,x_{k-1},1).$$
 Another application of $\Delta$ cancels all the remaining factors with $k-1\not\in B$ and yields
\[ \Delta^2 g =  \prod_{\{k-1,k\}\subseteq B\subseteq [k], |B| = 2mi, i\in [n]}  w^{(B)}_B. \]
 After $j\leq k$ applications of $\Delta$ we have
\[ \Delta^j g =  \prod_{\{k-j+1,\dots,k\}\subseteq B\subseteq [k], |B| = 2mi, i\in [n]} w^{(B)}_B. \]
 But for $j=k$ the above product is empty, which shows~\eqref{eq:Dkg}.

 Finally let $g\in W$ satisfy the assumption of the implication~\eqref{eq:TCk}, i.e.,
\[ g(x0)=g(x1) \text{ for all } x\in [2]^{k-1}\setminus \{(1,\dots,1)\}. \]
Then
\[ (\Delta g)(x) = \begin{cases} g(1,\dots,1,0)^{-1} g(1,\dots,1,1) & \text{if } x=(1,\dots,1), \\ 1 & \text{else.} \end{cases} \]
 After $j\leq k$ applications of $\Delta$ we obtain a $(k-j)$-ary function $\Delta^j g$ mapping $(1,\dots,1)$ to  
 $g(1,\dots,1,0)^{-1} g(1,\dots,1,1)$ and every other $x\in [2]^{k-j}$ to $1$.
 Thus $\Delta^k g$ is the constant function with value $g(1,\dots,1,0)^{-1} g(1,\dots,1,1)$. Since $\Delta^k g$ is constant $1$
 by~\eqref{eq:Dkg}, we have
\[ g(1,\dots,1,0) = g(1,\dots,1,1). \]
Thus every $g\in W$ satisfies~\eqref{eq:TCk}.
This concludes the proof that $S$ is supernilpotent of class at most $k-1=2mn$ and of the lemma.
\end{proof}

\begin{proof}[Proof of Theorem~\ref{thm:csi}]
 Follows from Lemmas~\ref{lem:csi} and~\ref{lem:csib}.
\end{proof}

 Corollary~\ref{cor:er} follows readily from the next auxiliary result which is a straighforward generalization of 
 \cite[Proposition 3.3.3, (4)$\Rightarrow$(1)]{Ho:FST}.
 
 \begin{lem}   \label{lem:er}
 Let $S$ be an eventually regular semigroup with a primitive idempotent $e$.
 Then $e$ generates a completely simple ideal of $S$. 
\end{lem}

\begin{proof}
 We show that $J_e$ is a minimal $\mathcal{J}$-class.
 Let $a\in S$ (not necessarily idempotent) such that  $J_a \leq J_e$.
 Since $S$ is eventually regular, we have $n\geq 1$ and $b\in S$ such that $a^n b a^n = a^n$.
 So $f := a^n b$ is idempotent and $J_f \leq J_a \leq J_e$.
 As in the proof of \cite[Proposition 3.3.3, (4)$\Rightarrow$(1)]{Ho:FST}, it now follows that $J_f = J_e$
 and consequently $J_a = J_e$.
 Hence $J_e$ is a minimal $\mathcal{J}$-class and $e$ generates a simple ideal $K$ of $S$.
 Since $e$ is primitive, $K$ is completely simple.
\end{proof}

\section{Monoids} \label{sec:monoid}

 Let $S$ be a solvable, nilpotent or supernilpotent monoid. Then $1$ is the unique idempotent in $S$
 by Example~\ref{ex:sl}. It follows that any solvable (left nilpotent, right nilpotent or supernilpotent)
 eventually regular monoid is a solvable (nilpotent) group.

 Mal'cev~\cite{Ma:NS} and independently Neumann and Taylor~\cite{NT:SNG} characterized subsemigroups of
 nilpotent groups by equations in the signature $\cdot$ without the inverse $^{-1}$.
 The latter used the following semigroup words. For variables $x,y$ and $z := (z_1,z_2,\dots)$ a sequence of countably
 many variables, let
\[ q_1(x,y,z) := xy,\quad q_{n+1}(x,y,z) := q_n(x,y,z)\, z_n\, q_n(y,x,z) \]
 for $n\geq 1$. Note that $q_n(x,y,z)$ only contains $x,y,z_1,\dots,z_{n-1}$.

 By~\cite[Theorem 1]{NT:SNG} a semigroup $S$ embeds into an $n$-nilpotent group iff $S$ is cancellative and
\[ S\models q_n(x,y,z) \approx q_n(y,x,z). \] 
 For $n\geq 1$ and a semigroup $S$, we define the congruence $\rho_n$ on $S$ generated by all pairs $(q_n(x,y,z), q_n(y,x,z))$,
\[ \rho_n := \Cg( (q_n(x,y,z),q_n(y,x,z)) \ |\ x,y\in S, z\in S^\N). \]

\begin{lem} \label{lem:monoidright}
 Let $S$ be a monoid and $n\in\N$. Then
 \[ (q_n(x,y,z),q_n(y,x,z)) \in [1)^{n+1} \text{ for all } x,y\in S, z\in S^\N. \]
\end{lem}

\begin{proof}
 We use induction on $n$. For $n=1$ we have to show that $(xy,yx)\in [1,1]$ for all $x,y\in S$.
 Consider
\[ \begin{bmatrix} x & 1 \\ x & 1 \end{bmatrix} \begin{bmatrix} 1 & 1 \\ y & y \end{bmatrix} \begin{bmatrix} 1 & x \\ 1 & x \end{bmatrix} = \begin{bmatrix} x & x \\ xy & yx \end{bmatrix} \in M(1,1). \]   
 Since the entries in the first row of the last matrix are equal, the entries in the second row $(xy,yx)$ are in
 $[1,1]$ by~\eqref{eq:tc}.

 Next let $x,y\in S, z\in S^\N$ and assume $(q_n(x,y,z),q_n(y,x,z)) \in [1)^{n+1}$ for fixed $n$. In $M([1)^{n+1},1)$ consider
\begin{align*} & \begin{bmatrix} q_n(x,y,z)z_n & 1 \\ q_n(x,y,z)z_n & 1 \end{bmatrix} \begin{bmatrix} q_n(x,y,z) & q_n(x,y,z) \\ q_n(y,x,z) & q_n(y,x,z) \end{bmatrix} \begin{bmatrix} 1 & z_n q_n(x,y,z) \\ 1 & z_n q_n(x,y,z) \end{bmatrix} \\
  = & \begin{bmatrix} q_n(x,y,z)z_nq_n(x,y,z) & q_n(x,y,z)z_n q_n(x,y,z) \\ q_n(x,y,z)z_nq_n(y,x,z) & q_n(y,x,z)z_n q_n(x,y,z) \end{bmatrix}.
\end{align*}
 Since the entries in the first row of the last matrix are equal and the entries in the second row are
 $(q_{n+1}(x,y,z),q_{n+1}(y,x,z))$, the latter is in $[1)^{n+2}$ by~\eqref{eq:tc}. The induction is complete.
\end{proof}

\begin{cor} \label{cor:rnqn}
 If $S$ is a right $n$-nilpotent monoid, then
\[ S \models q_n(x,y,z) \approx q_n(y,x,z). \]
\end{cor}

 The converse of the corollary is not true since there are commutative monoids (e.g. non-trivial semilattices)
 that are not abelian (i.e. right $1$-nilpotent).

\begin{lem} \label{lem:monoidleft}
 Left nilpotent monoids and supernilpotent monoids are cancellative.
\end{lem} 

\begin{proof}  
 Let $S$ be a monoid and let $a,c,d\in S$ such that $ac = ad$. From
\[ \begin{bmatrix} a & a \\ 1 & 1 \end{bmatrix} \begin{bmatrix} c & d \\ c & d \end{bmatrix} = \begin{bmatrix} ac & ad \\ c & d \end{bmatrix} \in M(1,1) \] 
we see $(c,d) \in [1,1]$ by~\eqref{eq:tc}. Hence
\[ \begin{bmatrix} ac & ad \\ c & d \end{bmatrix} \in M(1,[1,1]), \] 
 which yields $(c,d)\in [1,[1,1]]$. Iterating this argument yields $(c,d)\in (1]^{n+1}$ for all $n\in\N$.
 If $S$ is left $n$-nilpotent, then $c=d$, i.e. $S$ is cancellative.

 We can also extend the previous argument to higher commutators in a straightforward way. To simplify notation we just
 show the ternary case.
\[ \begin{bmatrix} a & a \\a & a \\ 1 & 1 \\ 1 & 1 \end{bmatrix}\begin{bmatrix} a & a \\ 1 & 1 \\ a & a \\ 1 & 1 \end{bmatrix} \begin{bmatrix} c & d \\ c & d \\ c & d \\ c & d \end{bmatrix} = \begin{bmatrix} a^2c & a^2d \\  ac & ad \\ ac & ad \\c & d \end{bmatrix} \in M(1,1,1) \] 
 yields $(c,d)\in [1,1,1]$ by~\eqref{eq:htc}. Similarly $(c,d)\in [1,\dots,1]$ for every $n$-ary higher commutator with $n\geq 2$.
 Hence, if $S$ is supernilpotent, then it is cancellative.
\end{proof}

Recall that a congruence $\sigma$ of a semigroup is \emph{cancellative} if $S/\sigma$ is cancellative.
Right nilpotent monoids are not necessarily cancellative (in particular not left nilpotent) as the following example
shows.

\begin{exa}
  Let $S$ be the monoid with presentation
\[ S := \langle x,y \st xy=yx=y^2=x^2 \rangle. \]
 Its elements have normal forms $1,y,x,x^2,x^3,\dots$.
 Hence $S$ is commutative but not cancellative. 
 Its smallest cancellative congruence $\gamma$ has a single non-trivial class $\{ x,y\}$.
 We show that $S$ is right $2$-nilpotent but not left nilpotent.

 {\bf Claim $(1]^{n+1} = \gamma$ for all $n\geq 1$:} First show that $(x,y)\in (1]^{n+1}$ for all $n\geq 0$ by induction.
 The base case for $n=0$ holds since $(1]^1 = 1$.
 Assume $n\geq 1$ in the following. By the induction hypothesis we have $(x,y)\in (1]^{n}$ and
\[ \begin{bmatrix} x & x \\ 1 & 1 \end{bmatrix} \begin{bmatrix} x & y \\ x & y \end{bmatrix} = \begin{bmatrix} x^2 & x^2 \\ x & y \end{bmatrix} \in M(1,(1]^n). \] 
 Then $(x,y) \in [1,(1]^n] = (1]^{n+1}$ by~\eqref{eq:tc} and the induction is proved.

 We showed $(1]^{n+1} \supseteq \gamma$ for all $n\geq 0$.
 For $n\geq 1$ the converse inclusion follows since $S/\gamma \cong (\N,+)$ is abelian.
 
 {\bf Claim $[[1,1],1] = 0$:} By the previous claim $M([1,1],1)$ is generated by
\[ A := \left\{ \begin{bmatrix} x & x \\ y & y \end{bmatrix}, \begin{bmatrix} y & y \\ x & x \end{bmatrix} \right\} \cup \left\{ \begin{bmatrix} c & d \\ c & d \end{bmatrix} \st c,d\in S\right\}. \]
By the presentation of $S$
\[ M([1,1],1) \setminus A = \left\{ \begin{bmatrix} xc & xd \\ yc & yd \end{bmatrix}, \begin{bmatrix} yc & yd \\ xc & xd \end{bmatrix} \st c,d\in S, c \neq d \right\}. \]
 For every element in $A$, equality in the first row clearly implies equality in the second row as well.
 For $\begin{bmatrix} xc & xd \\ yc & yd \end{bmatrix}$ where $c\neq d$, equality in the first row implies that
 $\{c,d\} = \{x,y\}$ and $xc=xd = x^2$. But then $yc=x^2=xd$ and we have equality in the second row as well.
 The case of $\begin{bmatrix} yc & yd \\ xc & xd \end{bmatrix}$ is clearly symmetric.
 Hence  $[[1,1],1] = 0$ as claimed.

 {\bf Claim $\underbrace{[1,\dots,1]}_{n+1} = \gamma$ for all $n\geq 1$:} This follows from a straightforward generalization
 of the case for $n=1$.
\end{exa}

\begin{proof}[Proof of Theorem~\ref{thm:nmonoid}]
 \eqref{it:crn}$\Rightarrow$\eqref{it:eng} follows from Corollary~\ref{cor:rnqn} and~\cite[Theorem 1]{NT:SNG}. 

 \eqref{it:eng}$\Rightarrow$\eqref{it:rln} is immediate.

 \eqref{it:rln}$\Rightarrow$\eqref{it:crn} follows from Lemma~\ref{lem:monoidleft}.
\end{proof}

\begin{cor} \label{cor:1n}
  Let $S$ be a monoid. Then the cancellative congruence generated by $[1)^{n+1}$
  (or simply by $\{(q_n(x,y,z),q_n(y,x,z)) \st x,y\in S, z\in S^\N \}$)
  is the smallest congruence $\delta$ of $S$ such that $S/\delta$ embeds into an $n$-nilpotent group. 
\end{cor}

 Again for $n=1$ this yields:
 for a monoid $S$, the commutator $[1,1]$ is the smallest congruence $\delta$ such that $S/\delta$ embeds into an
 abelian group. Equivalently, $[1,1]$ is the cancellative congruence of $S$ that is generated by
 $\{(xy,yx) \st x,y\in S\}$.

 Each of the equivalent conditions in Theorem~\ref{thm:nmonoid} implies $n$-supernilpotence.
 For $n\leq 2$ we can also prove the converse. The general case remains open.

\begin{thm} \label{thm:2sn}
 For $n\in\{0,1,2\}$ and a monoid $S$ the following are equivalent:
\begin{enumerate}
\item \label{it:sn2}
 $S$ is $n$-supernilpotent.
\item \label{it:rln2}
 $S$ is left and right $n$-nilpotent.
\item \label{it:eng2}
 $S$ embeds into an $n$-nilpotent group.  
\end{enumerate}
\end{thm}  

\begin{proof}
 Clearly $S$ is $0$-supernilpotent iff $S$ is trivial iff $S$ is left and right $0$-nilpotent. Further $S$ is $1$-supernilpotent iff
 it is abelian iff it is left and right $1$-nilpotent. In either case the result follows from Theorem~\ref{thm:nmonoid}.

 Hence we assume $n=2$ in the following. The implication \eqref{it:eng2}$\Rightarrow$\eqref{it:sn2} is immediate.
 For  \eqref{it:sn2}$\Rightarrow$\eqref{it:eng2}, assume $S$ is $2$-supernilpotent.
 Then $S$ is cancellative by Lemma~\ref{lem:monoidleft}. By~\cite[Theorem 1]{NT:SNG} it suffices to show that
 $S \models q_2(x,y,z) \approx q_2(y,x,z)$. Let $x,y\in S$. Then
 \[ \begin{bmatrix} xy & 1 \\ xy & 1 \\ xy & 1 \\ xy & 1 \end{bmatrix}
   \begin{bmatrix} x & x \\ x & x \\ 1 & 1 \\ 1 & 1 \end{bmatrix}
   \begin{bmatrix} 1 & 1 \\ y & y \\ 1 & 1 \\ y & y \end{bmatrix} 
   \begin{bmatrix} 1 & 1 \\ 1 & 1 \\ x & x \\ x & x \end{bmatrix}
   \begin{bmatrix} 1 & 1 \\ x & x \\ 1 & 1 \\ x & x \end{bmatrix}
   \begin{bmatrix} 1 & yx \\ 1 & yx \\ 1 & yx \\ 1 & yx \end{bmatrix}
 = \begin{bmatrix} xyx & xyx \\ xyxyx & xyxyx \\ xyx & xyx \\xy^2x^2 & yx^2yx \end{bmatrix} \in M(1,1,1) \]  
 and equality in the top three rows of the final matrix yields $xy^2x^2= yx^2yx$ by~\eqref{eq:htc}.
 Since $S$ is cancellative, we obtain $xy^2x=yx^2y$. 

 Next consider
 \[ \begin{bmatrix} x & 1 \\ x & 1 \\ x & 1 \\ x & 1 \end{bmatrix}
   \begin{bmatrix} 1 & 1 \\ y & y \\ 1 & 1 \\ y & y \end{bmatrix}
   \begin{bmatrix} 1 & x \\ 1 & x \\ 1 & x \\ 1 & x \end{bmatrix}
   \begin{bmatrix} 1 & 1 \\ 1 & 1 \\ z & z \\ z & z \end{bmatrix}
   \begin{bmatrix} 1 & x \\ 1 & x \\ 1 & x \\ 1 & x \end{bmatrix}
   \begin{bmatrix} 1 & 1 \\ y & y \\ 1 & 1 \\ y & y \end{bmatrix}
   \begin{bmatrix} x & 1 \\ x & 1 \\ x & 1 \\ x & 1 \end{bmatrix}
 = \begin{bmatrix} x^2 & x^2 \\ xy^2x & yx^2y \\ xzx & xzx \\ xyzyx & yxzxy \end{bmatrix} \in M(1,1,1). \]  
 Since $xy^2x=yx^2y$ and we have equality in the top three rows of the final matrix, we obtain $xyzyx = yxzxy$,
 that is $q_2(x,y,z) = q_2(y,x,z)$ by~\eqref{eq:htc}. This shows~\eqref{it:eng2} and completes the proof of the theorem.
\end{proof}
  
\begin{cor}
 Let $S$ be a monoid. Then
\begin{enumerate}
\item
 $[1,1,1]$ is the smallest congruence $\delta$ of $S$ such that $S/\delta$ embeds into a $2$-nilpotent group. 
\item 
 $[1,1,1]$ is the cancellative congruence of $S$ generated by $[[1,1],1]\vee [1,[1,1]]$, by $[[1,1],1]$ or
 simply by $\rho_2$.
\end{enumerate}
\end{cor}

\begin{proof}
 From its definition $[1,1,1]$ is the smallest congruence $\delta$ of $S$ such that $S/\delta$ is $2$-supernilpotent.
  
 Since $\bar{S} := S/[1,1,1]$ is $2$-supernilpotent, it embeds into a $2$-nilpotent group by Theorem~\ref{thm:2sn}.
 Equivalently $\bar{S}$ is left and right $2$-nilpotent. Thus $\bar{S}$ is the greatest cancellative quotient of $S/(1]^3$
 and of $S/[1)^3$.
\end{proof}  

 Finally we collect some results on commutators on free monoids.
 As it turns out free monoids are left $2$-nilpotent but in general not right nilpotent.
 
\begin{thm} \label{thm:free}
 Let $S$ be the free monoid over $\Sigma$. Then 
\begin{enumerate}
\item \label{it:free11}
 $[1,1] = \rho_1$.
\item \label{it:freeleft}
 $[1, [1,1]] = 0$.
\item \label{it:freeright}
 If $|\Sigma|\leq 2$, then $\rho_2 = [[1,1],1] = [1,1,1]$ is cancellative; else $\rho_2$ is not cancellative.
\item \label{it:freen}
 $S$ is right nilpotent iff $|\Sigma|\leq1$.
\end{enumerate}
\end{thm}

\begin{proof}
 \eqref{it:free11}
 The inclusion $[1,1]\supseteq\rho_1$ holds by Lemma~\ref{lem:monoidright}.
 For the converse note that $S/\rho_1$ is isomorphic to the free commutative monoid over $\Sigma$.
 In particular $S/\rho_1$ embeds into the free abelian group over $\Sigma$. Hence $S/\rho_1$ is abelian by
 Theorem~\ref{thm:nmonoid}. Thus $[1,1]\subseteq\rho_1$.  

 \eqref{it:freeleft} Let $u\in M(1,[1,1])$. Then we have $a,b\in S^k$,
 $c := (c_1,\dots,c_\ell), d := (d_1,\dots,d_\ell)\in S^\ell$ such that $c_i\, [1,1]\, d_i$ for all $i\in[\ell]$,
 and a semigroup term such that
\[ u =   \begin{bmatrix} t(a,c) & t(a,d) \\ t(b,c) & t(b,d) \end{bmatrix}. \]
 Assume the entries in the top row are equal, i.e.  $t(a,c) = t(a,d)$.
 By~\eqref{it:free11} the assumption $c_i\, [1,1]\, d_i$ implies that the lengths of $c_i$ and $d_i$ as words over
 $\Sigma$ are equal. Together with $t(a,c)=t(a,d)$ in the free monoid, this yields that actually $c_i=d_i$ for all $i\in [\ell]$.
 Thus $t(b,c) = t(b,d)$ and $[1,[1,1]]=0$.

\eqref{it:freeright}
 If $|\Sigma| \leq 1$, then $S$ is free commutative and the assertions are trivial.
 Next we show that for $\Sigma = \{x,y\}$ of size $2$ that 
\begin{equation} \label{eq:Srhocancellative}
 \rho_2 \text{ on } S=\{x,y\}^* \text{ is cancellative.}
\end{equation}
 Lallement provides the following normal forms for elements in $S/\rho_2$~\cite[Proposition 4.4]{La:NRF}:
 every $w\in S$ is $\rho_2$-equivalent to a unique word of the form
 $x^\alpha y^\gamma x^{\alpha'}$ or $x^\alpha y^\beta x y^{\beta'}x^{\alpha'}$ for some $\alpha,\alpha',\gamma\in\N$ and
 $\beta,\beta'>0$.

 From this description it follows that if $w$ is in normal form, then so is $xw$.
 Hence for $b,c\in S$ in normal form, $xb\equiv_{\rho_2} xc$ actually implies $xb=xc$ and further $b=c$
 since $S$ is cancellative. More generally for all $b,c\in S$ we see that $xb\equiv_{\rho_2} xc$ implies
 $b\equiv_{\rho_2} c$.
 Considering the normal forms for swapped generators $y,x$ yields that we can cancel $y$ in $S/{\rho_2}$
 as well. This proves~\eqref{eq:Srhocancellative}.

 By~\cite{NT:SNG} we know that $S/\rho_2$ embeds into a $2$-nilpotent group.
 Hence $S/\rho_2$ is right $2$-nilpotent and $[[1,1],1]\subseteq\rho_2$ by Theorem~\ref{thm:nmonoid}.
 The converse inclusion holds by Lemma~\ref{lem:monoidright}.
 With Theorem~\ref{thm:2sn} we also obtain that the higher commutator $[1,1,1]=\rho_2$.

 For $|\Sigma| \geq 3$ Shneerson shows that $\Sigma^*/\rho_2$ has intermediate growth~\cite[Theorem 6.1]{Sh:RFS}.
 Hence this monoid cannot embed into the free $2$-nilpotent group over $\Sigma$, which is well-known to
 have polynomial growth~\cite{Wo:GFG}.
 Thus $\rho_2$ is not cancellative by Theorem~\ref{thm:nmonoid}.

\eqref{it:freen}
 If $S$ embeds into any nilpotent group, then $|\Sigma|\leq 1$ by~\eqref{it:freeright} and Theorem~\ref{thm:nmonoid}.
 The converse is clear.
 Hence the statement follows from Theorem~\ref{thm:nmonoid}.
\end{proof}

\bibliographystyle{plain}
\def\cprime{$'$}

\end{document}